\documentclass{amsart}
\usepackage{amssymb,amsmath,amsthm,hyperref}

\usepackage[normalem]{ulem}

\theoremstyle{theorem}
\newtheorem{theorem}{Theorem}
\newtheorem{corollary}[theorem]{Corollary}

\title{Scaled Arndt Compositions}

\author{Brian Hopkins}
\address{Saint Peter's University, Jersey City, NJ 07306, USA}
\email{bhopkins@saintpeters.edu}

\author{Augustine Munagi}
\address{University of the Witwatersrand, Johannesburg 2000, South Africa}
\begin{document}

\maketitle

\begin{abstract}
Integer compositions restricted by inequalities on certain pairs of parts were first considered by J\"{o}rg Arndt in 2013 and several variations have been studied recently.  Here we consider a broad two-parameter generalization that scales the defining relation.  We connect these to compositions restricted to parts from certain congruence classes and establish the recurrence relations satisfied by the related counting sequences.  This provides new combinatorial interpretations to several documented integer sequences and also simple unrecorded sequences.  We conclude with suggestions for further study.
\end{abstract}

\section{Introduction}

A composition of a positive number $n$ is an ordered sequences of positive integers $(c_1, \ldots, c_t)$ called parts whose sum is $n$.  The study of compositions dates to ancient India at least with connections to the Fibonacci sequence defined by $F_0 = 0$, $F_1 = 1$, and $F_n = F_{n-1} + F_{n-2}$ for $n \ge 2$ \cite{h26}.

In 2013, J\"{o}rg Arndt observed that there are $F_n$ compositions of $n$ that satisfy $c_{2i-1} > c_{2i}$ for each $i$, although the first proof was published in 2022 by the first named author and Tangboonduangjit \cite{ht22}.  Those authors also considered the condition $c_{2i-1} > c_{2i} + k$ for an arbitrary integer $k$ \cite{ht22,ht23}.  Several other generalizations and variations of Arndt's conditions have been studied \cite{cr24,ht25,nacr25,p23}. 

Here, we consider a broad generalization of Arndt's condition, $s c_{2i-1} \ge t c_{2i}$ for positive integers $s$ and $t$.  We establish a bijection to compositions restricted to parts in particular congruence classes which leads to a recurrence relation.  We close by mentioning the related condition $s c_{2i-1} \ge t c_{2i} + k$ for an arbitrary integer $k$.

\section{Scaled Arndt compositions}
Given positive integers $s$ and $t$, let $A_{s,t}(n)$ be the compositions of $n$ that satisfy $s c_{2i-1} > t c_{2i}$ for each $i$; for a composition with an odd number of parts, the inequality is vacuously true for the final ``pair.''  Clearly $A_{s,t}(n) = A_{ks,kt}(n)$ for each positive integer $k$, so we may assume that $s$ and $t$ are relatively prime.  Let $a_{s,t}(n) = | A_{s,t}(n) |$.  Arndt's original claim in this notation is $A_{1,1}(n) = F_n$.

For example, 
\[A_{2,3}(6) = \{(6), (5, 1), (4, 2), (4, 1, 1), (3, 1, 2), (2, 1, 3), (2, 1, 2, 1)\};\]
note that $A_{1,1}(6)$ also includes $(3,2,1)$ but $2\cdot3 \not> 3\cdot 2$.  As an example of a more permissive set, note that $(3,4) \in A_{3,2}(7)$ since $3\cdot3 > 2\cdot4$ while the increasing pair $(3,4) \notin A_{1,1}(7)$.

\begin{theorem} \label{big}
For relatively prime positive integers $s$ and $t$, the count $a_{s,t}(n)$ equals the number of compositions of $n$ into parts congruent to \[r+ \left\lceil \frac{rt+1}{s} \right\rceil \bmod (s+t)\] for each $0 \le r \le s-1$.
\end{theorem}

\begin{proof}
Given a composition in $A_{s,t}(n)$, for each pair $(c_{2i-1},c_{2i}) = (a,b)$, write $b=qs+r$ with $0 \le r \le s-1$.  The condition $s c_{2i-1} > t c_{2i}$ implies $a \ge qt + \lceil (rt+1)/s \rceil$.  Make the assignment
\[(a,b) \mapsto \left(1^{a-qt - \lceil (rt+1)/s\rceil}, q(s+t) + r + \left\lceil\frac{rt+1}{s}\right\rceil \right)\]
which preserves the sum $a+qs+r$.  These images can have any number of parts 1 and at most one larger part which is congruent to $r + \lceil (rt+1)/s\rceil \bmod (s+t)$ for $1 \le r \le s-1$.  (Note that, for any $s$ and $t$, the value for $r=0$ is 1.)

Given a composition with parts satisfying the modulo $(s+t)$ congruence, decompose it into subsequences of the form $(1^c, d)$ where $c \ge 0$ (that is, there may be no parts 1 between larger parts) and $d = q(s+t) + (r + \lceil (rt+1)/s \rceil)$ for some $r$ with $1 \le r \le s-1$.  Make the assignment
\[(1^c, d) \mapsto \left( c + qt + \left\lceil\frac{rt+1}{s}\right\rceil, qs + r \right)\]
which preserves the sum $c+d$ and produces a pair satisfying $s c_{2i-1} > t c_{2i}$.  If the final subsequence consists entirely of $c$ parts 1, then its image is the singleton $c$.  Combining these pairs (and possible singleton) gives the scaled Arndt composition.

It is clear that the maps are inverses and thus establish a bijection.
\end{proof}

For example, $(s,t) = (2,3)$ gives parts congruent modulo 5 to
\[0 + \left\lceil \frac{0\cdot3 + 1}{2} \right\rceil = 1  \text{ or } 1 + \left\lceil \frac{1\cdot3+1}{2} \right\rceil = 3.\]
Write $C_{1,3(5)}(n)$ for the compositions of $n$ restricted to parts congruent to 1 or 3 modulo 5.  Table \ref{big23} shows the correspondence between $A_{2,3}(6)$ and $C_{1,3(5)}(6)$ given by Theorem \ref{big}.

\begin{table}
\caption{The $n=6$ bijection of Theorem \ref{big} using the shorthand notation $1^j$ for $j$ consecutive parts 1.} \label{big23}
\centering
\setlength{\tabcolsep}{5pt}
\renewcommand{\arraystretch}{1.5}
\begin{tabular}{c|ccccccc}
$A_{2,3}(6)$ & $(6)$ & $(5, 1)$ &  $(4, 2)$ &  $(4, 1, 1)$ &  $(3, 1, 2)$ &  $(2, 1, 3)$ &  $(2, 1, 2, 1)$ \\ \hline
$C_{1,3(5)}(6)$ & $(1^6)$ & $(1^3,3)$ & $(6)$ & $(1,1,3,1)$ & $(1,3,1,1)$ & $(3,1^3)$ & $(3,3)$
\end{tabular}
\end{table}

Table \ref{bigex} gives the list of congruence classes and moduli given by Theorem \ref{big} for small relatively prime $s$ and $t$.

\begin{table}[t]
\centering
\caption{Examples of congruence classes and moduli given by Theorem \ref{big} results.  For instance, $a_{2,3}(n)$ equals the number of compositions with parts congruent to 1 or 3 modulo 5 as detailed in Table \ref{big23}.} \label{bigex}
\setlength{\tabcolsep}{5pt}
\renewcommand{\arraystretch}{1.25}
\begin{tabular}{c|ccccc}
$s \backslash t$ & 1 & 2 & 3 & 4 & 5  \\ \hline
1 & $1$ (2) & $1$ (3) & $1$ (4) & $1$ (5) & $1$ (6)  \\
2 & $1,2$ (3) & & $1,3$ (5) & & $1,4$ (7)  \\
3 & $1,2,3$ (4) & $1,2,4$ (5) & & $1,3,5$ (7) & $1,3,6$ (8)  \\
4 & $1,2,3,4$ (5) & & $1,2,4,6$ (7) & &  $1,3,5,7$ (9)  \\
5 & $1,2,3,4,5$ (6) & $1,2,3,5,6$ (7) & $1,2,4,5,7$ (8) & $1,2,4,6,8$ (9) &  
\end{tabular}
\end{table}

The bijection to compositions with parts determined by congruence classes in Theorem \ref{big} allows us to determine a recurrence for each sequence $a_{s,t}(n)$.

\begin{corollary} \label{rec}
For relatively prime positive integers $s$ and $t$, the count $a_{s,t}(n)$ satisfies
\[a_{s,t}(n) = \left( \sum_{r=1}^{s-1} a_{s,t}\!\left(n-r- \left\lceil\frac{rt+1}{s}\right\rceil \right) \right) + a_{s,t}(n-s-t).\]
\end{corollary}

\begin{proof}
The generating function for the number of compositions with parts restricted to a set $A$ is $1 / (1 - \sum_{a \in A} x^a)$ \cite[Thm. 3.13]{hm10}.

Let $m_r = r + \lceil (rt+1)/s \rceil$ so that our set $A$ is all parts congruent to $m_r$ modulo $(s+t)$ for some $r$ satisfying $0 \le r \le s-1$.  Then 
\[ \sum_{a \in A} x^a = \frac{\sum_{r=0}^{s-1} x^{m_r}}{1-x^{s+t}} \]
so that the generating function for $a_{s,t}(n)$ is 
\[ \frac{1}{1 - \sum_{a \in A} x^a} = \frac{1}{1 - \frac{\sum_{r=0}^{s-1} x^{m_r}}{1-x^{s+t}}} = \frac{1-x^{s+t}}{1-x^{s+t} - \sum_{r=0}^{s-1} x^{m_r}}\]
whose denominator establishes the recurrence.
\end{proof}

Note that evaluating the first $(s+t)$ terms of the generating function in the proof of Corollary \ref{rec} gives the initial values of the recurrence.  (Since there are $2^{n-1}$ compositions of $n$ and we know $m_0 = 1$ is always an allowed part, $1 \le a_{s,t}(n) \le 2^{n-1}$ for all positive $n, s, t$ with $s$ and $t$ relatively prime.)

As examples of the generating functions developed in the proof of Corollary \ref{rec}, here are the cases where $s+t = 5$.  Recall the convention that there is one empty composition of zero.
\begin{align*}
\sum & \, a_{1,4}(n) q^n  = \sum c_{1(5)}(n) q^n  = \frac{1-q^5}{1-q-q^5} \\
& = 1 + q + q^2 + q^3 + q^4 + q^5 + 2 q^6 + 3 q^7 + 4 q^8 + 5 q^9 +  \cdots \\[1ex]
\sum & \, a_{2,3}(n) q^n  = \sum c_{1,3(5)}(n) q^n  = \frac{1-q^5}{1-q-q^3-q^5} \\
& = 1 + q + q^2 + 2 q^3 + 3 q^4 + 4 q^5 + 7 q^6 + 11 q^7 + 17 q^8 +  27 q^9 +  \cdots \\[1ex]
\sum & \, a_{3,2}(n) q^n  = \sum c_{1,2,4(5)}(n) q^n  = \frac{1-q^5}{1-q-q^2-q^4-q^5} \\
& = 1 + q + 2 q^2 + 3 q^3 + 6 q^4 + 10 q^5 + 19 q^6 + 34 q^7 + 62 q^8 +  112 q^9 + \cdots \\[1ex]
\sum & \, a_{4,1}(n) q^n  = \sum c_{1,2,3,4(5)}(n) q^n  = \frac{1-q^5}{1-q-q^2-q^3-q^4-q^5} \\
& = 1 + q + 2 q^2 + 4 q^3 + 8 q^4 + 15 q^5 + 30 q^6 + 59 q^7 + 116 q^8 +  228 q^9 + \cdots
\end{align*}

Several sequences $a_{s,t}(n)$ for particular $s$ and $t$ are found in the On-Line Encyclopedia of Integer Sequences.  In addition to the Fibonacci sequence, the sequences $a_{1,t}$ are \cite[A000930, A003269, A003520, A005708]{o} for $2 \le t \le 5$.  The sequences $a_{s,1}$ are, with offsets, \cite[A001590, A001631, A124312, A251706]{o} for $2 \le s \le 5$, sometimes called $s$-bonacci sequences.  The only other sequence from Table \ref{bigex} currently listed is $a_{3,4}(n)$ \cite[A117761]{o} with no combinatorial interpretation.  In all cases, the interpretation in terms of scaled Arndt compositions is new.  Table \ref{missing} gives sequences not currently listed in \cite{o}.

\begin{table}[h]
\centering
\caption{Sequences $a_{s,t}(n)$ with relatively prime $s$ and $t$ not currently in the On-Line Encyclopedia of Integer Sequences \cite{o} up to $s+t = 8$.} \label{missing}
\setlength{\tabcolsep}{5pt}
\renewcommand{\arraystretch}{1.25}
\begin{tabular}{c|rrrrrrrrrr}
$n$ & 1 & 2 & 3 & 4 & 5 & 6 & 7 & 8 & 9 & 10\\ \hline
$a_{2,3}(n)$ & 1 & 1 & 2 & 3 & 4 & 7 & 11 & 17 & 27 & 42 \\
$a_{3,2}(n)$ & 1 & 2 & 3 & 6 & 10 & 19 & 34 & 62 & 112 & 203 \\ \hline
$a_{2,5}(n)$ & 1 & 1 & 1 & 2 & 3 & 4 & 5 & 8 & 12 & 17 \\
$a_{4,3}(n)$ & 1 & 2 & 3 & 6 & 10 & 19 & 33 & 61 & 109 & 198 \\
$a_{5,2}(n)$ & 1 & 2 & 4 & 7 & 14 & 27 & 51 & 99 & 190 & 365 \\ \hline
$a_{3,5}(n)$ & 1 & 1 & 2 & 3 & 4 & 7 & 11 & 16 & 26 & 41\\
$a_{5,3}(n)$ & 1 & 2 & 3 & 6 & 11 & 20 & 37 & 67 & 124 & 227
\end{tabular}
\end{table}

\section{Affine scaled Arndt compositions}

A natural next step in the study, following \cite{ht22,ht23} for Arndt's original compositions, is to consider compositions satisfying
$ s c_{2i-1} > t c_{2i} + k$
for each positive $i$ where $s$ and $t$ are relatively prime positive integers and $k$ is any integer.  We have established a few partial results, but encourage research into a unified treatment of this problem.

\end{document}